\documentclass[12pt]{article}

\usepackage{amsmath}
\usepackage{amssymb}
\usepackage{amsthm}
\usepackage{amsfonts}
\usepackage{graphicx}
\usepackage{enumerate}
\usepackage{mathrsfs}
\usepackage{bbm}
\usepackage{hyperref}

\hypersetup{colorlinks=true, linkcolor=blue, urlcolor=blue,
linktocpage=true, implicit=false, pdfstartview=FitH,
pdfpagemode=None}

\newtheoremstyle{nonum}{}{}{\itshape}{}{\bfseries}{.}{ }{\thmnote{#3}}

\newtheorem{thm}{Theorem}[section]

\newtheorem{cor}[thm]{Corollary}

\newtheorem{prop}[thm]{Proposition}

\theoremstyle{definition}
\newtheorem{exam}[thm]{Example}
\newtheorem{defn}[thm]{Definition}
\newtheorem{rem}[thm]{Remark}

\newtheorem{problem}[thm]{Problem}

\theoremstyle{nonum}

\DeclareMathOperator{\dist}{dist}

\newcommand{\tit}{\textit}

\newcommand{\mcal}{\mathcal}

\newcommand{\al}{\alpha}
\newcommand{\be}{\beta}

\newcommand{\eps}{\varepsilon}

\newcommand{\ph}{\varphi}
\newcommand{\om}{\omega}

\newcommand{\ze}{\zeta}

\newcommand{\R}{\mathbb R}
\newcommand{\C}{\mathbb C}
\newcommand{\Z}{\mathbb Z}
\newcommand{\N}{\mathbb N}

\newcommand{\const}{\equiv}

\newcommand{\pois}[1]{\{#1\}} 

\newcommand{\br}[1]{\bigl({#1}\bigr)}

\newcommand{\imp}{\Rightarrow}

\DeclareMathOperator{\ad}{ad}

\newcommand{\cA}{{\mathcal{A}}}

\newcommand{\cE}{{\mathcal{E}}}
\newcommand{\cF}{{\mathcal{F}}}

\newcommand{\cH}{{\mathcal{H}}}

\newcommand{\cK}{{\mathcal{K}}}

\newcommand{\cP}{{\mathcal{P}}}

\newcommand{\tHam}{{\widetilde{\hbox{\it Ham}}\, (M,\om)}}
\newcommand{\Ham}{{\hbox{\it Ham}}\, (M,\om)}

\begin{document}

\title{Poisson brackets, quasi-states and symplectic integrators}

\renewcommand{\thefootnote}{\alph{footnote}}

\author{\textsc Michael Entov$^{a}$,\ Leonid
Polterovich$^{b}$, Daniel Rosen $^{b}$}

\footnotetext[1]{Partially supported by the Israel Science
Foundation grant $\#$ 881/06.} \footnotetext[2]{Partially supported
by the Israel Science Foundation grant $\#$ 509/07.}

\date{\today}

\maketitle

\begin{abstract}
This paper is a fusion of a survey and a research article. We
focus on certain rigidity phenomena in function spaces associated
to a symplectic manifold. Our starting point is a lower bound
obtained in an earlier paper with Zapolsky for the uniform norm of
the Poisson bracket of a pair of functions in terms of symplectic
quasi-states. After a short review of the theory of symplectic
quasi-states, we extend this bound to the case of iterated Poisson
brackets. A new technical ingredient is the use of symplectic
integrators. In addition, we discuss some applications to
symplectic approximation theory and present a number of open
problems.
\end{abstract}

\tableofcontents

\vfill\eject

\section{Introduction and main results}

We discuss certain aspects of function theory on symplectic
manifolds related to the so-called {\it $C^0$-rigidity of the
Poisson bracket} which can be described as follows.

Let $(M, \omega)$ be a closed symplectic manifold. Denote by
$C^\infty (M)$ the space of smooth functions on $M$ and by $\|\cdot
\|$ the standard {\it uniform norm}  (also called the {\it
$C^0$-norm}) on it: $ \| F\| := \max_{x\in M} |F(x)|$. The Poisson
bracket on $C^\infty (M)$, induced by $\omega$, will be denoted by
$\{\cdot, \cdot\}$. Applying the Poisson bracket repeatedly we get
so-called {\it iterated Poisson brackets} of functions from
$C^\infty (M)$.

Note that the definition of an iterated Poisson bracket of two
smooth functions $F, G\in C^\infty (M)$ involves partial
derivatives of the functions. Thus \emph{a priori} one does not
expect any restrictions on possible changes of the Poisson bracket under $C^0$-small
perturbations of $F$ and $G$. Surprisingly, such
restrictions do exist (see \cite{CV, EPZ, Zap1, Hum, EP-Poisson1,
EP-Poisson2, Buh}).

In this paper we will discuss such restrictions coming from the
theory of {\it symplectic quasi-states} and our starting point is
a lower bound on the $\| \{ F,G\}\|$ in terms of certain
symplectic quasi-states obtained in \cite{EPZ}. We recall that
symplectic quasi-states are functionals on $C^{\infty} (M)$
introduced in \cite{EP-qs} which obey a convenient set of axioms
of an algebraic flavor and, furthermore, are Lipschitz with
respect to the uniform norm, so that the above-mentioned lower
bound is robust with respect to $C^0$-perturbations of the
functions and thus provides a restriction of the needed sort. Let
us also note that when $\dim M \geq 4$, only currently known
symplectic quasi-states come from Floer theory, see
Example~\ref{exam-qm} below for a brief discussion.

In the present paper we extend the bound from \cite{EPZ} to the
case of iterated Poisson brackets. A new technical ingredient of
the paper is an unexpected use of {\it symplectic integrators}
whose origins lie in numerical Hamiltonian dynamics.

In addition, we discuss some applications to symplectic
approximation theory, which are in the spirit of \cite{EPZ} but
presented under a somewhat different angle, and formulate a number
of open problems.

\subsection{Quasi-morphisms and quasi-states}

A \emph{homogeneous quasi-morphism} on a group $G$ is a function
$\mu \colon G \to \R$ satisfying
\begin{enumerate}[$(i)$]
\item There exists a constant $C$ such that for any $g,h \in G$
\begin{equation*}
|\mu(gh) - \mu(g) - \mu(h) | \leq C.
\end{equation*}
\item For any $g \in G$ and $n \in \Z$ one has $\mu(g^n) = n
\mu(g)$.
\end{enumerate}

\medskip
\noindent Nowadays quasi-morphisms are intensively studied due to
their importance in  group theory as well as their applications in
geometry and dynamics (see e.g. \cite{Kotschick} for a brief
introductory survey).

In this paper we are concerned with quasi-morphisms on a certain
important group appearing in symplectic geometry. Let us recall
its definition along with a few preliminaries (see \cite{MS,Pbook}
for more details).

Given a time-dependent Hamiltonian $H: M\times [0,1]\to\R$, denote
$H(t):= H(\cdot,t): M\to\R$. We say that $H(t)$ is {\it
normalized} if it has zero mean for all $t$. The set of all
normalized (time-dependent) Hamiltonian functions is denoted by
$\cF$. Symplectomorphisms of $(M,\omega)$ which can be included in
a Hamiltonian flow (i.e. the flow of a time-dependent Hamiltonian
vector field) generated by a Hamiltonian from $\cF$ form a group
$\Ham$. Let $\tHam$ be the universal cover of $\Ham$ with the base
point at the identity. This is precisely the group of interest for
us and we will consider quasi-morphisms on this group. Given a
(time-dependent) Hamiltonian $H$ on $M$, we denote by $\phi^t_H$
the Hamiltonian flow generated by $H$. We write $\phi_H$ for the
element of $\tHam$ represented by the path $\{\phi^t_H\}$, $t \in
[0,1]$.

\medskip
\noindent{\bf Convention:} Let $\text{proj}: \tHam \to \Ham$ be the
the natural projection. In what follows for the sake of brevity we
write $F \circ \phi$ instead of $F \circ \text{proj}(\phi)$, where
$\phi \in \tHam$ and $F \in C^{\infty}(M)$.

\medskip

A homogeneous quasi-morphism $\mu$ on $\tHam$ is called
\emph{stable}, if for any (time-dependent) Hamiltonian functions
$F, G\in\cF$
\begin{equation}
\int_0^1 \min_M (F(t) - G(t)) d t \leq \mu(\phi_G) - \mu(\phi_F)
\leq \int_0^1 \max_M (F(t) - G(t)) d t\;.
\end{equation}
In particular, for a stable $\mu$ one has
\begin{equation}
\label{eqn-mu-Lipschitz} \|\mu(\phi_F) - \mu(\phi_G)\| \leq
\int_0^1 \| F(t) - G(t)\| dt.
\end{equation}

Though the definition of a stable quasi-morphism is fairly simple,
it is not easy to prove its existence. Nowadays such quasi-morphisms
are constructed for certain symplectic manifolds (for instance, for
complex projective spaces and their products) with the help of the
Floer theory. However, for instance, it is unknown whether such a
quasi-morphism exists on the standard symplectic torus  -- see
\cite{EP-toric} for a detailed discussion of symplectic manifolds
admitting quasi-morphisms constructed by means of the Floer theory.

\medskip
\noindent
\begin{exam}\label{exam-qm} {\rm
To give the reader a feeling of what we are talking about, let us
very briefly outline construction of a stable quasi-morphism on
$\tHam$ where $(M,\omega)$ is the complex projective space $\C
P^n$ equipped with the standard Fubini-Study form, normalized so
that the total symplectic volume of $\C P^n$ is 1. Let $\Lambda$
be a covering of the free loop space of $M$ whose elements are
equivalence classes of pairs $(\gamma, u)$, where $\gamma: S^1 \to
M$ is a loop and $u:D^2 \to M$ is a disc with $u|_{\partial
D^2}=\gamma$. The equivalence relation is defined as follows:
$(\gamma_1,u_1)\sim (\gamma_2,u_2)$ whenever $\gamma_1 =\gamma_2$,
while the discs $u_1$ and $u_2$ are homotopic with fixed boundary.
Every Hamiltonian $H \in \cF$ defines {\it the classical action
functional}
$$\cA_H : \Lambda \to \R,\;\; [(\gamma, u)] \mapsto \int_0^1
H(\gamma(t),t) dt - \int_{D^2} u^*\omega\;.$$ Its critical points
correspond to $1$-periodic orbits of the Hamiltonian flow
$\phi^t_H$, or in other words to fixed points of its time-one map.
Floer theory is an infinite-dimensional version of the
Morse-Novikov homology theory for the action functional $\cA_H$ on
$\Lambda$.  In particular, given $\alpha \in \R$, it enables one
to define Floer homology groups of the sublevel sets $\{\cA_H <
\alpha\}$ which we denote $HF(\alpha)$. These groups come with
natural inclusions $j_{\alpha}: HF(\alpha) \to HF(\infty)$. The
group $HF(\infty)$ can be identified via the
Piunikhin-Salamon-Schwarz isomorphism with another remarkable
object, the quantum homology algebra of $M$. This is an algebra
with the unity element $e$. The next construction (due to Oh,
Schwarz and Viterbo) is a Floer homological analogue of the
standard min-max: put
$$c(H):= \inf \{\alpha\;:\; e \in \text{Image}(j_\alpha)\}\;.$$
It turns out that the value $c(H)$ depends only on the time-one
map $\phi_H \in \tHam$ and not on the Hamiltonian $H$ itself.
Writing $c(\phi_H):=c(H)$, we get a map $c: \tHam \to \R$. The
specific algebraic structure of the quantum homology algebra  of
$\C P^n$ guarantees that $$|c(\phi\psi)-c(\phi)-c(\psi)| \leq
\text{const}\;\; \forall \phi,\psi \in \tHam\;$$ for some constant
independent of $\phi, \psi$. Therefore the homogenization (with
the opposite sign)
$$\mu(\phi) :=- \lim_{k \to \infty} \frac{c(\phi^k)}{k}$$
of the function $\phi\mapsto c(\phi)$ is a homogeneous
quasi-morphism. A deep fact of the Floer theory is that $\mu$ is
stable. We refer the reader to \cite{MS} and to references therein
for the relevant preliminaries on Floer theory and quantum
homology.}
\end{exam}

\medskip

Any stable homogeneous quasi-morphism $\mu$ on $\tHam$ induces a
functional $\ze \colon C^{\infty}(M) \to \R$ as follows. From now on
assume for simplicity that the symplectic volume $\int_M \om^n$ of
$M$ equals $1$. For $F \in C^{\infty} (M)$, set
\begin{equation*}
\ze(F) = \int_M F \om^n - \mu(\phi_F).
\end{equation*}

The functional $\ze$ -- which is, in general, {\it non-linear} --
satisfies the following system of axioms \cite{EPZ}:
\begin{enumerate}[(i)]
\item $\ze(1) = 1$.
\item\label{qs_monoton_axiom} $F \geq G \imp \ze(F) \geq \ze(G)$.
\item\label{qs_linear_axiom} $\ze(aF + bG) = a \ze(F) + b\ze(G)$ for any $F,G \in C^{\infty}(M)$
such that $\pois{F,G} \const 0$ (that is $F,G$ commute with respect
to the Poisson bracket) and any $a, b \in \R$.
\end{enumerate}

\medskip
\noindent The functionals satisfying (i)-(iii) were introduced in
\cite{EP-qs} and are called {\it symplectic quasi-states}. An
important feature of symplectic quasi-states is the they are
Lipschitz in the uniform norm $||F||=\max |F|$ on $C^\infty(M)$:
\begin{equation}\label{eq-ze-Lipschitz}
\ze(F) = \ze(G + F - G) \leq \ze(G + ||F-G||) = \ze(G) + ||F-G||\;.
\end{equation}
The theory of symplectic quasi-states lies on the borderline
between symplectic geometry and functional analysis and its
origins lie in mathematical foundations of quantum mechanics.
In particular, any symplectic quasi-state extends to a
topological quasi-state in the sense of Aarnes \cite{Aarnes} on
the space of continuous functions $C(M)$: This means that it is
linear on any singly-generated subalgebra of $C(M)$. For
further discussion on quasi-morphisms and quasi-states on
symplectic manifolds, see \cite{EP-qs, EP-toric, EPZ, EPZ-physics,
Zap1, Zap2}.

\medskip
\noindent \begin{exam}\label{exam-qs}{\rm In the case when $M = \C
P^1 =S^2$ and the symplectic form is the standard area form
normalized so that the total area of $S^2$ is 1, the quasi-state
$\zeta$ associated to the stable quasi-morphism from
Example~\ref{exam-qm} can be described in elementary combinatorial
terms. Let $F$ be a generic Morse function on $M$. Its {\it Reeb
graph} $\Gamma$ is obtained by collapsing connected components of
the level sets of $F$ to points. It is easy to see that $\Gamma$
is a tree. Let $\sigma$ be the push-forward of the symplectic area
from $M$ to $\Gamma$. One can show that there exists unique point
$m \in \Gamma$, called {\it the median}, so that $\sigma(Y) \leq
1/2$ for every connected component $Y$ of $\Gamma \setminus m$. It
turns out that $\zeta(F)$ equals to the value of $F$ on the
connected component of its level set which corresponds to $m$.}
\end{exam}

\medskip

Given  a quasi-state $\ze$, we can define a functional
$$
\Pi(F,G) = |\ze(F+G) - \ze(F) - \ze(G)| $$ on the space $\cH:=
C^{\infty}(M) \times C^{\infty}(M)$. Let $d$ be the uniform distance
on  $\cH$ :
$$d((F,G),(F',G')):= ||F-F'||+||G-G'||\;.$$
It follows from \eqref{eq-ze-Lipschitz} that $\Pi$ is Lipschitz with
respect to $d$:
\begin{equation}\label{Pi_Lipschitz}
|\Pi(F,G) - \Pi(F',G')| \leq 2d((F,G),(F',G'))\;.
\end{equation}

\medskip
\noindent {\it From now on we assume that $\zeta$ is a symplectic
quasi-state associated to a stable homogeneous quasi-morphism on}
$\tHam$.

\subsection{Lower bound on higher Poisson brackets}

There is another class of functionals on $\cH:= C^{\infty}(M)
\times C^{\infty}(M)$ of a more classical nature. Denote by
$\mcal{P}_N$ the set of Lie monomials in two variables involving
$N$-times-iterated Poisson brackets (i.e. $\mcal{P}_1$ consists
 of $\pois{F,G}$, $\mcal{P}_2$ of $\pois{\pois{F,G},F}$ and $\pois{\pois{F,G},G}$ and so on).
 For $F,G \in C^{\infty}(M)$ set
\begin{equation*}
Q_N(F,G) = \sum_{p \in \mcal{P}_{N-1}} || p(F,G)||.
\end{equation*}

Our main result is:
\begin{thm}\label{main_thm}
Let $\ze \colon C(M) \to \R$ be a symplectic quasi-state induced by a stable homogeneous quasi-morphism, and let
\begin{equation*}
\Pi(F,G) = |\ze(F+G) - \ze(F) - \ze(G)|.
\end{equation*}
Then there exist constants $C_N$ for any $N\in\N$ so that
\begin{equation}\label{new_bound}
\Pi(F,G) \leq C_N \cdot Q_N(F,G)^{1/N}
\end{equation}
for any $F,G \in C^{\infty}(M)$.
\end{thm}

\medskip
\noindent Theorem~\ref{main_thm} is another display of
$C^0$-rigidity of (iterated) Poisson brackets mentioned in the
introduction: since the definition of $p(F,G)$, $p\in\cP_{N-1}$,
involves partial derivatives of functions $F$ and $G$, there are no
\emph{a priori} restrictions on the changes of $Q_N(F,G)$ under
perturbations of $F$ and $G$ in the uniform norm. At the same time
the functional $\Pi$ is Lipschitz in the uniform norm, and hence
inequality \eqref{new_bound} yields such a restriction.

For the case $N=2$ inequality \eqref{new_bound} had been proved in
an earlier paper \cite{EPZ}.

\subsection{Application to symplectic approximation
theory}\label{subsec-app}

The discussion below was initiated in \cite{EPZ}, though here we
slightly change the viewpoint. The problem we are going to deal with
can be roughly stated as follows.

\medskip
\noindent \begin{problem}\label{prob-appr} Given a pair of
functions on a symplectic manifold, what is its optimal uniform
approximation by a pair of (almost) Poisson-commuting functions?
\end{problem}

\medskip
\noindent The value $\Pi(F,G)$ can be considered as a fancy
measure of non-com\-mu\-ta\-ti\-vi\-ty of functions $F,G \in
C^{\infty}(M)$. We illustrate this as follows: Note that
inequality \eqref{Pi_Lipschitz} implies that
$$\Pi(F,G) \leq 2\max(||F||,||G||)\;.$$ This motivates the following definition:
We say that a pair of functions $F,G$ is $\zeta$-{\it extremal} if
$||F||=||G||=1$ and the previous inequality is an equality:
$$\Pi(F,G) = 2\;.$$ We denote by $\cE \subset \cH$ the subset of all
$\zeta$-extremal pairs and by $\cK\subset \cH$ the set of all
Poisson-commuting pairs: $\cK := \{ (F,G)\in \cH\ |\ \{
F,G\}\equiv 0\}$.

\medskip
\noindent \begin{exam} \label{exam-sphere}{\rm Let $M = S^2 \times
\ldots \times S^2$ be the product of $n$ copies of the unit sphere
$S^2 =\{x^2 +y^2 +z^2 =1\}$ in $\R^3$. Each sphere is equipped
with the standard area form $\theta$, normalized so that the total
area of $S^2$ is 1, and the symplectic structure on $M$ equals
$\theta \oplus \ldots \oplus \theta$. Denote by $x_i,y_i,z_i$ the
standard Euclidean coordinates on the $i$-th factor of $M$. It is
known \cite{EP-qs} that $M$ admits a symplectic quasi-state
$\zeta$ associated to a stable quasi-morphism so that the
functions $F = 1-2x_1^2$ and $G=1-2y_1^2$ form a $\zeta$-extremal
pair. In the case when $n=1$, the quasi-state $\zeta$ is described
in Example~\ref{exam-qs} above. }\end{exam}

\medskip

Our starting observation is that
\begin{equation}\label{eq-d-1}
d((F,G),\cK)=1\;\; \forall (F,G) \in \cE\;.
\end{equation}

\medskip
\noindent Indeed, if $F'$ and $G'$ Poisson-commute we have that
$\Pi(F',G') = 0$. Given any $\zeta$- extremal pair $(F,G)$, we
apply inequality \eqref{Pi_Lipschitz}  and get
$$2 \leq 2d((F,G),(F',G'))\;,$$
and thus $d((F,G),\cK) \geq 1$. Since $(F,0) \in \cK$, we get the
opposite inequality, and thus \eqref{eq-d-1} follows.

\medskip

The set $\cK$ possesses a natural system of ``tubular
neighborhoods"
$$\cK_N(\epsilon) := \{(F,G) \in \cH \;:\; Q_N(F,G) <
\epsilon\}\;.$$ This viewpoint is justified by  a symplectic
version of the Landau-Hadamard-Kolmogorov inequality proved in
\cite{EP-Poisson2} which implies (for $N\geq 2$) that
$$\| \{ F,G\}\|\leq {\rm const} (N)\cdot \min (||F||,||G||)^{\frac{N-2}{N-1}}
\cdot \|p(F,G)\|^{\frac{1}{N-1}},$$ for {\it some} $p\in
\cP_{N-1}$. Since $\|p(F,G)\|\leq Q_N(F,G)$, we get that for any
integer $N \geq 2$ there exists a constant $a_N
>0$ so that
\begin{equation}\label{eq-KHL}
||\{F,G\}|| \leq a_N \cdot \min (||F||,||G||)^{\frac{N-2}{N-1}}
\cdot Q_N(F,G)^{\frac{1}{N-1}}\;\; \forall F,G \in
C^{\infty}(M)\;.\end{equation} In particular, $(F,G) \in \cK$
provided $Q_N(F,G) = 0$.

Finally, we arrive to the following problem: given an extremal
pair $(F,G)$, explore the behavior of the function
$$D_N(\epsilon) := d((F,G), \cK_N(\epsilon))$$
as $\epsilon \to 0$. Noticing that for some real $b_N = b_N(F,G)$,
and all positive $\epsilon < 1$ the pair $(F,\epsilon b_N G)$ lies
in $\cK_N(\epsilon)$, we get an obvious upper bound for
$D_N(\epsilon)$:
\begin{equation}\label{eq-upper}
D_N(\epsilon) \leq 1- b_N\epsilon  \;\;\forall 0< \epsilon < 1\;.
\end{equation}

\medskip
\noindent As far as the lower bound is concerned, we claim that
\begin{equation}\label{eq-lower}
D_N(\epsilon) \geq 1- \frac{1}{2}C_N\epsilon^{\frac{1}{N}} \;\; \forall
\epsilon
> 0 \;,
\end{equation}
where $C_N$ is the constant from Theorem~\ref{main_thm}.

\medskip
\noindent This inequality immediately follows from the following
corollary of Theorem~\ref{main_thm} (cf. \cite{EPZ} for the case
$N=2$).

\begin{cor}\label{lem-Q}
For any $F,G,F',G' \in C^{\infty}(M)$ we have
\begin{equation*}
  \frac{\Pi(F,G)}{2}-d((F,G),(F',G')) \leq \frac{1}{2}C_N Q_N(F',G')^{\frac{1}{N}}\;.
\end{equation*}
\end{cor}

\begin{proof} By inequality \eqref{Pi_Lipschitz},
$$\Pi(F',G') \geq \Pi(F,G) - 2d((F,G),(F',G'))\;.$$
Applying Theorem~\ref{main_thm} to $(F',G')$ we get the desired
inequality.
\end{proof}

\subsection{Discussion and open problems}

The gap between upper and lower bounds \eqref{eq-upper} and
\eqref{eq-lower} suggests the following question.

\medskip
\noindent
\begin{problem}\label{prob-1}
What is the actual asymptotical behavior of function
$D_N(\epsilon)$ as $\epsilon \to 0$? The answer is unknown even
for the specific extremal pair of functions described in Example
\ref{exam-sphere} above.
\end{problem}

\medskip

Another question related to inequality \eqref{eq-lower} is as
follows:

\medskip
\noindent
\begin{problem}\label{prob-2} Given a specific extremal pair of functions
(say, as in Example~\ref{exam-sphere} above), can one prove lower
bound \eqref{eq-lower} for $D_N(\epsilon)$ without methods of
``hard" symplectic topology (like Floer theory)? (See
\cite{Gromov} for a discussion about the dichotomy between ``hard"
and ``soft" symplectic topology).
\end{problem}

\medskip
\noindent Let is emphasize once again that our proof of
\eqref{eq-lower} uses the fact that the symplectic manifold $M=S^2
\times \ldots \times S^2$ admits a symplectic quasi-state
associated to a stable quasi-morphism on $\widetilde{Ham}(M)$
which follows from Floer-homological considerations.

Let us also note that in dimension ${\rm dim}\, M =2$ there exist
alternative constructions of symplectic quasi-states (see e.g.
\cite{Aarnes}, \cite{EP-qs}, \cite{Zap1}, \cite{Zap2},
\cite{Rosenberg}) which do not involve Floer homology. None of
those quasi-states is known to be induced by a stable
quasi-morphism. For instance, it is shown in \cite{Rosenberg} that
Py's quasi-morphism \cite{Py} gives rise to a quasi-state, but it
is unknown whether this quasi-morphism is stable. On the other
hand, Zapolsky \cite{Zap1,Zap2} proved that for a wide class of
quasi-states $\zeta$ on surfaces one has inequality
$$ \Pi(F,G):= |\zeta(F+G)-\zeta(F)-\zeta(G)| \leq \sqrt{||{F,G}||_{L_1}}\;.$$
Note that this is a sharper version of inequality
\eqref{new_bound} in Theorem~\ref{main_thm} for $N=2$, where the
uniform norm is replaced by the $L_1$-norm. Interestingly enough,
Zapolsky's argument does not involve quasi-morphisms: it is based
on methods of two-dimensional topology. This discussion leads to
the following problem.

\medskip
\noindent
\begin{problem}\label{prob-2-a} Can one extend Zapolsky's inequality to the case of iterated Poisson
brackets? More precisely, given a closed 2-dimensional symplectic
manifold $(M,\omega)$ and a quasi-state $\zeta$,  is it true that
for $N \geq 3$
$$\Pi(F,G) \leq \text{const}(N) \cdot \big{(}\sum_{p \in
\mcal{P}_{N-1}} ||p(F,G)||_{L_1}\Big{)}^{\frac{1}{N}}\;?$$ In case
the answer is affirmative, it would be interesting to develop an $L_1$-version of symplectic approximation theory on surfaces along the lines of Section~\ref{subsec-app} above.
\end{problem}

\medskip

Our current impression is that Theorem~\ref{main_thm} lies on a
rather narrow border-line between ``soft" and ``hard'' symplectic
topology, but on the ``hard" side. To illustrate this, let us
compare the inequalities
$$\Pi(F,G) \leq C_N \cdot
Q_N(F,G)^{1/N}\;\;\;\; (\spadesuit_N)$$ for the previously known
case $N=2$ and the new case $N \geq 3$. Assume that
$||F||=||G||=1$. By the symplectic Landau-Hadamard-Kolmogorov
inequality \eqref{eq-KHL} (which is proved by elementary calculus)
$$Q_2(F,G)^{\frac{1}{2}} \leq a'_N Q_N(F,G)^{\frac{1}{2N-2}}\;.$$
Thus inequality $(\spadesuit_2)$ yields
$$\Pi(F,G) \leq C'_N Q_N(F,G)^{\frac{1}{2N-2}}\;.$$
Since $2N-2 > N$ for $N \geq 3$, the latter inequality is weaker
than $(\spadesuit_N)$ provided $N \geq 3$ and $Q_N(F,G)$ is small
enough . In fact, it can be  shown \cite{Rosen} that on every
symplectic manifold of dimension $\geq 4$ there exist sequences of
functions $A_\epsilon$ and $B_\epsilon$ so that
$||A_\epsilon||=||B_\epsilon|| = 1$ and
$$Q_N(A_\epsilon,B_\epsilon) = \beta_N \cdot \epsilon^{2N-2},\;$$
where $\beta_N, N \geq 2$, is a sequence of positive numbers, and
$\epsilon \in (0,\epsilon_0)$ for some $\epsilon_0 >0$ independent
of $N$. Thus for any fixed $N$ we have that
$Q_N(A_\epsilon,B_\epsilon)^{\frac{1}{N}} \sim
\epsilon^{2-\frac{2}{N}}$ as $\epsilon\to 0$. Therefore, for a
pair of integers $N
> L \geq 2$ the upper bound for $\Pi(A_\epsilon,B_\epsilon)$ given
by $(\spadesuit_N)$ is sharper than the one given by
$(\spadesuit_L)$ provided $\epsilon$ is small enough.

\medskip

Let us mention finally that the functionals $Q_N$ have been
studied within function theory on symplectic manifolds from a
different viewpoint. It is known that $Q_2$ and $Q_3$ are lower
semi-continuous with respect to the uniform metric $d$ on $\cH$,
and in fact have rather tame local behavior
\cite{EP-Poisson1,Buh,EP-Poisson2}. It would be interesting to
investigate the lower semi-continuity of $Q_N$ for $N > 3$.

\subsection{Symplectic integrators}\label{symp_int_sect}

Our proof of Theorem~\ref{main_thm} for $N \geq 3$ follows closely
the lines of \cite{EPZ} with one technical innovation: we use
symplectic integrators for the proof of our main theorem. Their
appearance was somewhat unexpected to us: symplectic integrators
have been designed for the purposes of numerical Hamiltonian
dynamics \cite{Yoshida}, the subject which is seemingly remote
from the theme of the present work.

\medskip
\noindent

Let $\Phi_t$, $\Psi_t$, $t \in (-\eps, \eps)$, be two smooth
families of diffeomorphisms of a closed manifold $M$. We say that
they are {\it equivalent modulo} $t^{N}$, denoted
\[
\Phi_t = \Psi_t \mod t^N,
\]
if for any $f \in C^{\infty} (M)$ and any $x \in M$
\begin{equation*}
|f(\Phi_t x) - f(\Psi_t x)| = O(t^N)\ {\rm as}\ t\to 0.
\end{equation*}
Equivalently, for any $x \in M$, one should have
$\dist\br{\Phi_t(x), \Psi_t(x)} = O(t^N)$ as $t\to 0$ for any
Riemannian metric defined around $x$.

\medskip
\noindent
\begin{defn}
A \emph{symplectic integrator} of order $N$ is a set of real
numbers $\al_1, \ldots, \al_m, \be_1, \ldots, \be_m$ such that for
any $F, G \in C^{\infty}(M)$,
\begin{equation}\label{symp_int_def}
\phi_{F+G}^t  = \phi_{\al_1 F}^t \, \phi_{\be_1 G}^t
\cdot\ldots\cdot \phi_{\al_m F}^t  \, \phi_{\be_m G}^t  \mod
t^{N+1}.
\end{equation}
\end{defn}

\medskip
\noindent Note that taking $G=0$ in \eqref{symp_int_def} yields
\begin{equation*}
\phi_{F}^t = \phi_{\sum_{i=1}^m \al_i F}^t  \mod t^{N+1}.
\end{equation*}
and hence $\sum_{i=1}^m \al_i = 1$. Similarly, $\sum_{i=1}^m \be_i = 1$.

\medskip
\noindent A crucial fact used in the next section is the {\it
existence of symplectic integrators of all orders}. This was
proved by Yoshida \cite{Yoshida} (see also \cite{Rosen} for a
detailed proof).

\section{Proof of the main theorem}

\subsection{Using symplectic integrators}\label{main_thm_sect}

We start with the following estimate. Let $\{\al_i, \be_i
\}_{i=1}^m$ be a symplectic integrator of order $N-1$, where $N
\geq 2$. Let $F,G \in C^{\infty}(M)$ be Hamiltonian functions with
zero mean. Set
\[
\Psi_N^{t} = \phi_{\al_1 F}^t \phi_{\be_1 G}^t \cdot\ldots\cdot
\phi_{\al_m F}^t \phi_{\be_m G}^t
\]
and note that, by the cocycle formula (see e.g. \cite{Pbook}), the
Hamiltonian flow $\Psi_N^t$ is generated by the Hamiltonian
\begin{align}\label{K_N_sum}
K_N &= \al_1 F + \be_1 G \circ \phi_{\al_1  F}^{-t} + \ldots +
\notag \\
& + \al_m F \circ \phi_{\be_{m-1} G}^{-t}\circ \phi_{\al_{m-1}
F}^{-t} \circ\ldots\circ
\phi_{\be_1  G}^{-t}\circ \phi_{\al_1  F}^{-t} + \notag \\
& + \be_m G \circ \phi_{\al_{m} F}^{-t}\circ \phi_{\be_{m-1}
G}^{-t}\circ \phi_{\al_{m-1} F}^{-t} \circ\ldots\circ \phi_{\be_1
G}^{-t}\circ \phi_{\al_1 F}^{-t}.
\end{align}

\medskip
\noindent
\begin{prop}\label{remainder-prop-1} There exists a constant $\kappa>0 $ independent of $F$ and
$G$ such that
\begin{equation}
\label{eqn-K-N-F-G} ||K_N (t) - (F+G)|| \leq \kappa \cdot Q_N(F,G)
t^{N-1}
\end{equation}
for all $t \in [0,1]$.
\end{prop}

\medskip
\noindent Together with \eqref{eqn-mu-Lipschitz}
Proposition~\ref{remainder-prop-1} immediately yields the following
corollary.

\begin{cor}
\label{cor-mu-integrator} Let $\mu: \tHam\to\R$ be a stable
homogeneous quasi-morphism. Then, under the hypothesis of
Proposition~\ref{remainder-prop-1},
\begin{equation}\label{symp_int_error}
|\mu(\phi_{F+G}) - \mu(\Psi_N^1)| \leq  k Q_N(F,G),
\end{equation}
where $k:= \frac{\kappa}{N}$.
\end{cor}

The proof of Proposition~\ref{remainder-prop-1} is given in
Section~\ref{estimating_remainder}.

\begin{proof}[Proof of Theorem~\ref{main_thm}]
Evidently it is enough to consider the case when $F$ and $G$ have
zero mean. Let $\Psi_N^t$ be as above and let $\mu$ be a stable
homogeneous quasi-morphism inducing the symplectic quasi-state
$\zeta$. Then
\[
\Psi_N^1 = \phi_{\al_1 F} \phi_{\be_1 G} \cdot\ldots\cdot
\phi_{\al_m F} \phi_{\be_m G}
\]
and hence, since $\mu$ is a quasi-morphism,
\[
|\mu(\Psi_N^1) - \sum_{i=1}^m \mu(\phi_{\al_i F}) -
\sum_{i=1}^m \mu(\phi_{\be_i G}) | \leq 2m \cdot C.
\]
Combining this with \eqref{symp_int_error} we obtain
\begin{multline*}
|\ze(F+G) - \sum_{i=1}^m \ze(\al_i G) - \sum_{i=1}^m \ze(\be_i G) | \\
 = |\mu(\phi_{F+G}) - \sum_{i=1}^m \mu(\phi_{\al_i F}) -
 \sum_{i=1}^m \mu(\phi_{\be_i G}) | \leq 2m \cdot C + k \cdot Q_N(F,G).
\end{multline*}
Note that
\begin{equation*}
\sum_{i=1}^m \ze(\al_i F) = \Bigl( \sum_{i=1}^m \al_i \Bigr) \ze(F) = \ze(F).
\end{equation*}
Similarly,
\begin{equation*}
\sum_{i=1}^m \ze(\be_i G) = \Bigl( \sum_{i=1}^m \be_i \Bigr) \ze(g) = \ze(G).
\end{equation*}
Therefore
\begin{equation}\label{unbalanced_ineqal}
\Pi(F,G) = |\ze(F+G) - \ze(F) - \ze(G)| \leq 2m \cdot C + k \cdot Q_N(F,G).
\end{equation}
Finally, note that, by the homogeneity of $\ze$, for any $E >0$
\[
\Pi(EF, EG) = E \cdot \Pi(F,G).
\]
On the other hand,
\[
Q_N(EF,EG) = E^N Q_N(F,G).
\]
Substituting both into \eqref{unbalanced_ineqal} and dividing by $E$ we obtain
\begin{equation*}
\Pi(F,G) \leq \frac{2m \cdot C}{E} + k Q_N(F,G)E^{N-1}.
\end{equation*}
The right-hand side is minimized by
\begin{equation*}
E_{min} = \Bigl( \frac{2m \cdot C}{k(N-1)Q_N(F,G)} \Bigr)^{1/N}
\end{equation*}
which yields the inequality
\begin{equation*}
\Pi(F,G) \leq C_N \cdot Q_N(F,G)^{1/N}
\end{equation*}
for
\begin{equation*}
C_N = N (N-1)^{\frac{1-N}{N}} (2mC)^{\frac{N-1}{N}}
k^{\frac{1}{N}}\;.
\end{equation*}

\end{proof}

\subsection{Estimating the remainder}\label{estimating_remainder}

In order to prove Proposition~\ref{remainder-prop-1}, we need to
write a finite order expansion for a composition of several
Hamiltonian flows. We use the following notation: Given smooth
functions $H_1, \ldots, H_n, A \in C^{\infty}(M)$ and non-negative
integers $i_1, \ldots, i_{n-1} $, such that $\sum_{j=1}^{n-1} i_j
\leq N -1$ denote $h_t^{(k)} := \phi_{H_k}^t$ for any $k$ and
$l:=N-\sum_{j=1}^{n-1} i_j$. We denote $\ad_H F:=\{F,H\}$. Set

\begin{multline*}
I_N^{(n)}(H_1, \ldots, H_n, i_1, \ldots, i_{n-1}, A, t) := \\
\int_0^t d s_1 \int_0^{s_1} d s_2 \ldots \int_0^{s_{l  - 1}}d s_l
(\ad_{H_n})^l (\ad_{H_{n-1}})^{i_{n-1}} \cdot\ldots\cdot
(\ad_{H_1})^{i_1} A \circ h_{s_l}^{(n)}
\end{multline*}
and
\[
I_N^{(1)}(H_1, A, t) = \int_0^t d s_1 \int_0^{s_1} d s_2 \ldots
\int_0^{s_{N - 1}} d s_N (\ad_{H_1})^N A \circ h_{s_N}^{(1)}.
\]

With this notation the expansion takes the following form:

\begin{prop}\label{composition_lemma}
Let $A, H_1, \ldots, H_n \in C^{\infty}(M)$.  Then
\begin{multline*}
A \circ h_t^{(1)} \circ \ldots \circ h_t^{(n)}=\\ = \sum_{i_1 +
\ldots + i_n \leq N-1} \frac{1}{i_1!} \cdot\ldots\cdot
\frac{1}{i_n !} (\ad_{H_n})^{i_n} \cdot\ldots\cdot
(\ad_{H_1})^{i_1} A \cdot t^{i_1 + \ldots + i_n} + R_N^{(n)},
\end{multline*}
where
\begin{multline*}
R_N^{(n)}=\\
=\sum_{i_1 + \ldots + i_{n-1} \leq N-1} \frac{1}{i_1 !}
\cdot\ldots\cdot \frac{1}{i_{n-1}!} I_N^{(n)}(H_1, \ldots, H_n,
i_1, \ldots, i_{n-1}, A, t) \cdot t^{\sum i_j} +\\
+R_N^{(n-1)} \circ h_t^{(n)}
\end{multline*}
and
\begin{align*}
R_N^{(1)} = I_N^{(1)}(H_1, A, t).
\end{align*}
\end{prop}

\medskip
\noindent \begin{rem}\label{rem-order}{\rm  While this formula may
seem complicated, its importance lies in the fact that for {\it
any} $n$ the term $I_N^{(n)}$, and hence the remainder
$R_N^{(n)}$, contains Lie monomials in variables
$H_1,\ldots,H_n,A$ involving $N$-times-iterated Poisson brackets.
Furthermore, recalling that $l=\sum_{j=1}^{n-1} i_j$, note that
$$I_N^{(n)}(H_1, \ldots, H_n, i_1, \ldots, i_{n-1}, A, t)=
O(t^l)\ {\rm as}\ t\to 0$$ because of the multiple integral of
multiplicity $l$ appearing in the definition of $I_N^{(n)}$. Hence,
for {\it any} $n$ we have (by induction) that $R_N^{(n)} = O(t^N)$
as $t\to 0$.}
\end{rem}

\medskip

\begin{proof}
The proof will proceed by induction on $n$. The case $n=1$ is
simply the Taylor expansion with the Lagrange remainder written as
a multiple integral:
\begin{align*}
A \circ h_t^{(1)} &= A+ \sum_{i=1}^{N-1} \frac{d^i}{dt^i} (A
\circ h_t^{(1)})(0) \frac{t^i}{i!} + \\
& + \int_0^t d s_1         \int_0^{s_1} d s_2 \ldots \int_0^{s_{N - 1}} d s_N \frac{d^N}{dt^N} (A \circ h_{s_N}^{(1)}) \\
    &= A + \sum_{i=1}^{N-1} (\ad_{H_1})^i A \cdot \frac{t^i}{i!} + \\
    & + \int_0^t d s_1 \int_0^{s_1} d s_2 \ldots \int_0^{s_{N - 1}} d s_N (\ad_{H_1})^N A     \circ h_{s_N}^{(1)} \\
    &= \sum_{i=0}^{N-1} \frac{1}{i!}(\ad_{H_1})^i A \cdot t^i + I_N^{(1)}(H_1, A, t).
\end{align*}
Now, assume the result holds for $n$. Then
\begin{align*}
A & \circ h_t^{(1)} \circ \ldots \circ h_t^{(n+1)} = \\
    & = \left(\sum_{i_1 + \ldots + i_n \leq N-1} \frac{1}{i_1!}
    \cdot\ldots\cdot \frac{1}{i_n !} (\ad_{H_n})^{i_n}
    \cdot\ldots\cdot (\ad_{H_1})^{i_1} A \cdot t^{i_1 + \ldots + i_n}
    + R_N^{(n)}  \right) \circ\\
    & \circ h_t^{(n+1)} =\\
    & =  \sum_{i_1 + \ldots + i_n \leq N-1} \frac{1}{i_1!}
    \cdot\ldots\cdot \frac{1}{i_n !} (\ad_{H_n})^{i_n} \cdot\ldots\cdot (\ad_{H_1})^{i_1} A
    \circ h_t^{(n+1)} \cdot t^{i_1 + \ldots + i_n} + \\
    & + R_N^{(n)} \circ h_t^{(n+1)}.
\end{align*}
Using the case $n=1$ we get
\begin{multline*}
(\ad_{H_n})^{i_n} \cdot\ldots\cdot (\ad_{H_1})^{i_1} A  \circ
h_t^{(n+1)}  = \\
= \sum_{i_{n+1}=0}^{N-1-i_1-\ldots -i_n} (\ad_{H_{n+1}})^{i_{n+1}}
(\ad_{H_n})^{i_n} \cdot\ldots\cdot
(\ad_{H_1})^{i_1} A \cdot \frac{t^{i_{n+1}}}{i_{n+1}!}+\\
 + I_{N-i_1- \ldots -i_n}^{(1)} \br{H_{n+1}, (\ad_{H_n})^{i_n}
\cdot\ldots\cdot (\ad_{H_1})^{i_1} A, t}.
\end{multline*}
Set $l=\sum_{j=1}^{n} i_j$ and note that
\begin{align*}
I&_{N-l}^{(1)} \br{H_{n+1}, (\ad_{H_n})^{i_n} \cdot\ldots\cdot (\ad_{H_1})^{i_1} A, t} =\\
    & = \int_0^t d s_1 \int_0^{s_1} d s_2 \ldots \int_0^{s_{N - l - 1}} d s_{N-l}
    (\ad_{H_{n+1}})^{N - l} (\ad_{H_n})^{i_n}
    \cdot\ldots\cdot\\
    & \cdot (\ad_{H_1})^{i_1} A \circ h_{s_{N-l}}^{(n+1)} =\\
    & = I_N^{(n+1)}(H_1, \ldots, H_{n+1}, i_1, \ldots, i_n, A, t).
\end{align*}
Therefore
\begin{align*}
A & \circ h_t^{(1)} \circ \ldots \circ h_t^{(n+1)} = \\
    &=  \sum_{i_1 + \ldots + i_n \leq N-1} \frac{1}{i_1!} \cdot\ldots\cdot \frac{1}{i_n
    !}\cdot\\
    & \cdot \left[ \sum_{i_{n+1}=0}^{N-l-1} (\ad_{H_{n+1}})^{i_{n+1}} (\ad_{H_n})^{i_n} \cdot\ldots\cdot (\ad_{H_1})^{i_1} A \cdot \frac{t^{i_{n+1}}}{i_{n+1}!} \right] \cdot t^{i_1 + \ldots + i_n} + \\
& \sum_{i_1 + \ldots + i_n \leq N-1} \frac{1}{i_1!} \cdot\ldots\cdot \frac{1}{i_n !}\cdot I_N^{(n+1)}(H_1, \ldots, H_{n+1}, i_1, \ldots, i_n, A, t) \cdot t^{i_1 + \ldots + i_n} +\\
& + R_N^{(n)} \circ h_t^{(n+1)} =\\
    &= \sum_{i_1 + \ldots + i_{n+1} \leq N-1} \frac{1}{i_1!} \cdot\ldots\cdot
    \frac{1}{i_{n+1}!}\cdot\\
    & \cdot (\ad_{H_{n+1}})^{i_{n+1}} (\ad_{H_n})^{i_n} \cdot\ldots\cdot
    (\ad_{H_1})^{i_1} A \cdot t^{i_1 + \ldots + i_{n+1}} +
    R_N^{(n+1)},
\end{align*}
which is the desired result.
\end{proof}

Finally, we need to relate the notions of equivalence modulo $t^N$
for Ha\-mil\-to\-ni\-an flows and Hamiltonian functions.

\begin{prop}\label{mod N true flows}
Let $U(t), V(t)$ be smooth time-dependent Hamiltonian fun\-ctions.
Then
\[
\phi_{U}^t = \phi_V^t \mod t^N \iff U(t) - V(t) = O(t^{N-1}) \
{\rm as}\ t\to 0.
\]
\end{prop}

\begin{proof}

For any function $H = H(t)$ depending on $t$ we will denote by
$H^{(i)}$ the $i$-th derivative of $H$ with respect to $t$. Then
$$\frac{d}{d t} \ph \circ \phi_U^t =
\pois{\ph,U(t)} \circ \phi_U^t \;,$$
$$\frac{d^2}{d t^2} \ph \circ \phi_U^t
= \pois{\ph, U^{(1)}(t)}\circ \phi_U^t + \pois{\pois{\ph,
U(t)},U(t)} \circ \phi_U^t\;$$ and, generally,
$$\frac{d^n}{d
t^n}\ph \circ \phi_U^t = (\pois{\ph, U^{(n-1)}(t)}+
S_n(U,\ph,t))\circ \phi_U^t\;,$$
where $S_n(U,\ph,t)$ is a Lie polynomial involving Poisson brackets of
$U^{(i)}(t)$ and $\ph$ for $i < n$ only.
Substituting $t=0$ we get that
\begin{equation}\label{eq-vsp-chain}
\left.\frac{d^n}{d t^n}\right|_{t=0}\ph \circ \phi_U^t = \pois
{\ph, U^{(n-1)}(0)}+ S_n(U,\ph,0)\;.
\end{equation}

By definition, $\phi_{U}^t = \phi_V^t \mod t^N$ if and only if for any $\ph
\in C^{\infty}(M)$ and any $n \leq N-1$,
\[
\left.\frac{d^n}{d t^n}\right|_{t=0}\ph \circ \phi_U^t =
\left.\frac{d^n}{d t^n}\right|_{t=0}\ph \circ \phi_V^t
\]%

Assume first that  $\phi_{U}^t = \phi_V^t \mod t^N$. Thus the
right-hand sides of the equations \eqref{eq-vsp-chain} for $U$ and
for $V$ coincide for all $n=1,\ldots,N-1$. Observe that if an
autonomous function  $H \in \mcal{F}$ satisfies $\pois{\ph,H}=0$
for all $\ph \in C^{\infty}(M)$, then $H \const 0$. Thus,
increasing $n$ from $1$ to $N-1$, we consecutively get that
\begin{equation}\label{eq-vsp-chain1}
U(0)=V(0),U^{(1)}(0) = V^{(1)}(0),\ldots, U^{(N-2)}(0) =
V^{(N-2)}(0)\;,\end{equation} which is equivalent to the fact that
$U(t) - V(t) = O(t^{N-1})$ as $t\to 0$.

Vice versa, \eqref{eq-vsp-chain1} yields that the right-hand sides
of the equations \eqref{eq-vsp-chain} for $U$ and for $V$ coincide
for all $n=1,\ldots,N-1$ and hence $\phi_{U}^t = \phi_V^t \mod
t^N$.
\end{proof}

\begin{proof}[Proof of Proposition~\ref{remainder-prop-1}]
By definition, $\Psi_N^t = \phi_{F+G}^t \mod t^{N}$, and hence, by
Proposition~\ref{mod N true flows}, $K_N = F+G \mod t^{N-1}$ as
$t\to 0$. Denote $H_N := K_N - (F+G)$. Thus $H_N (t) =
O(t^{N-1})$. Applying Proposition~\ref{composition_lemma} and
Remark~\ref{rem-order}, we get that $H_N$ is the sum of the
remainders $R_{N-1}$ of each term in the sum \eqref{K_N_sum}. In
turn, by Remark~\ref{rem-order}, each such remainder is the sum of
terms of the form $c t^{N-1} p(F,G) \circ h_t$, where $\{h_t\}$ is
a path of Hamiltonian diffeomorphisms, $p$ is a monomial from
$\mcal{P}_{N-1}$ and $c$ is a constant {\it independent of $F$ and
$G$}. Thus, recalling that
$$Q_N(F,G) = \sum_{p \in \mcal{P}_{N-1}} ||p(F,G)||,$$
we get that
$$\| K_N (t) - (F + G)\| = ||H_N (t)|| \leq  \kappa Q_N(F,G) t^{N-1}\;,$$
where $\kappa$ is a constant depending only on $N$ and on the
choice of the symplectic integrator. This finishes the proof of
the proposition.
\end{proof}


\bibliographystyle{alpha}

\begin{tabular}{l}
Michael Entov \\
Department of Mathematics\\
Technion - Israel Institute of Technology\\
Haifa 32000, Israel \\
entov@math.technion.ac.il \\
\end{tabular}

\medskip

\begin{tabular}{l}
Leonid Polterovich\\
School of Mathematical Sciences\\
Tel Aviv University\\
Tel Aviv 69978, Israel\\
and\\
Department of Mathematics\\
University of Chicago\\
Chicago, IL 60637, USA\\
polterov@runbox.com\\
\end{tabular}

\medskip

\begin{tabular}{l}
Daniel Rosen\\
School of Mathematical Sciences\\
Tel Aviv University\\
Tel Aviv 69978, Israel\\
da.rosen@gmail.com\\
\end{tabular}

\end{document}